\documentclass[12pt,reqno]{amsart}
\usepackage{amsfonts}
\usepackage{amsmath,amssymb}
\usepackage{marginnote}

\def\marg#1#2{\def\marginnotetextwidth{\the\textwidth}\marginnote{\bf #1}{\bf #2}}

\numberwithin{equation}{section}

\newtheorem{theorem}{Theorem}[section]
\newtheorem{theorem*}{Theorem}
\newtheorem{prop}[theorem]{Proposition}
\newtheorem{lemma}[theorem]{Lemma}
\newtheorem{cor}[theorem]{Corollary}

\newtheorem{definition}[theorem]{Definition}
\newtheorem*{definition*}{Definition}
\newtheorem{remark}[theorem]{Remark}

\newcommand{\C}{{\mathbb C}}
\newcommand{\Z}{{\mathbb Z}}

\newcommand{\GL}{{\operatorname{\bf GL_n}}}

\newcommand{\Ne}{{\mathcal{N}}}
\newcommand{\Oe}{{\mathcal{O}}}

\newcommand{\w}{{\operatorname{\bf w}}}

\newcommand{\di}{\operatorname{diag}}
\newcommand{\tr}{{\operatorname{tr}}}

\newcommand{\Cr}{{\operatorname{ Cr}}}

\newcommand{\Me}{{\operatorname{ M}}}

\newcommand{\Imm}{\operatorname{Im}}

\title[Double cosets $NgN$\dots]{Double cosets $Ng N$ of normalizers of maximal tori of  simple algebraic groups and  orbits of  partial actions of  Cremona subgroups}
\author{N. Gordeev }

\thanks{The research
is  financially supported by the grant of RFFI-19-01-00297. }

\address{ Department of Mathematics, Russian State Pedagogical University, Moijka 48, St.Petersburg, 191-186, Russia}

\email{nickgordeev@mail.ru}

\date{12.12.2021}

\author{E. Egorchenkova}

\address{Department of Mathematics,\\ Herzen State
Pedagogical University,\\ 48 Moika Embankment, \\191186, St.Petersburg,\\
RUSSIA}
 \email{e-egorchenkova@mail.ru}



\renewcommand{\GL}{\mathrm{GL}}
\newcommand{\SL}{\mathrm{SL}}

\def\<{\langle}
\def\>{\rangle}
\def\tilde{\widetilde}
\def\a{\alpha}
\def\b{\beta}
\def\phi{\varphi}
\def\Ue{\mathcal U}

\def\t{{\mathfrak t}}
   \def\w{{\bf w}}
\begin{document}
\maketitle

\parindent0pt

\begin{abstract}
Let $G$ be a simple algebraic group over an algebraically closed field $K$ and let $N = N_G(T)$ be the normalizer of a fixed maximal torus $T\leq G$.  Further, let $U$ be the unipotent radical of a fixed Borel subgroup $B$ that contains $T$ and let $U^-$ be the unipotent radical of the opposite Borel subgroup $B^-$. The Bruhat decomposition implies the decomposition $G = NU^-UN$. The Zariski closed subset $U^-U\subset G$ is isomorphic to the affine space $A_K^m$ where $m = \dim G -\dim T$ is the number of roots in the corresponding root system. Here we construct a  subgroup $\Ne\leq \Cr_m(K)$ that ``acts partially'' on $A_K^m\approx\Ue$ and we show that there is one-to-one correspondence between the orbits of such a partial action and the set of double cosets $\{NgN\}$.  Here  we also calculate the set $\{g_\a\}_{\a \in \mathfrak A}\subset \Ue$ in the simplest case $G = \SL_2(\C)$.

\end{abstract}

\section*{Introduction}

{\it Double cosets of transformation groups.}  Let $G$ be a group which acts on a set $X$. Further, let $x_1, x_2\in X, \,\,H_1 = \operatorname{St}_{ x_1}, H_2 = \operatorname{St}_{ x_2} $ be  stabilizers of $x_1, x_2$ and let $\Oe_{x_1}, \Oe_{x_2}$  be  orbits of $x_1, x_2$. Then there is one-to-one correspondence between the set of  $G$-orbits of the natural action on  $\Oe_{x_2}\times \Oe_{x_1}$ and the set of double cosets $\{H_1\, g_\a H_2\}_{\a\in \mathfrak A}$ (it is a simple and well-known fact). Namely,
$$\{(g_\a(x_2), x_1)\}_{\a\in \mathfrak A}\,\,\,\text{ is the minimal set of representations of  G-orbits of} \,\,  \Oe_{x_2}\times \Oe_{x_1}.\eqno(*)$$

\bigskip

{\it The  pairs of maximal tori of simple algebraic groups.}
In this paper we consider the case when $G$ is a simple algebraic group over an algebraically closed field $K$. The decomposition of   a group $G$ into the union of double cosets $G = \cup_{g_i} H_2 g_i H_1$ is a very important construction in the theory of algebraic groups, especially in the case when $H_1,  H_2$ are parabolic subgroups. For these cases the decomposition is finite. Here we consider the case  when $H_1 = H_2 = N = N_G(T)$ is the normalizer of a fixed maximal torus $T$. Now let $X$ be the set of all maximal tori of $G$. The group $G$ acts on $X$ by conjugation. Then $X$ is just one $G$-orbit of $T\in X$  and $N := N_G(T) = \operatorname{St}_T$. Thus, we have one-to-one correspondence between the set of  $G$-orbits of the set  $X\times X$  and the set of double cosets $\{N g_\a N\}_{\a \in \mathfrak A}$. Further, we have  the decomposition
$$G =  N U^-UN,\,\,\,\text{where}\,\, U = R_u(B), \,\,U^- = R_u(B^-)$$
(here $B$ is a fixed Borel subgroup that contains $T$,  $B^- = w_0 Bw_0$ is the opposite Borel subgroup, $R_u(B)$ is the unipotent radical of $B$). Note,
$$\Ue:= U^-U\approx A_K^m$$
where $m $ is the number of roots in the corresponding root system and $A_K^m$ is the $m$-dimensional affine space over $K$.
Thus we have one-to-one correspondence between the set of  $G$-orbits of $X\times X$  and the set
of double cosets  $\{N\mathfrak u N\}_{\mathfrak u \in \Ue}$.

\bigskip

{\it  The group $\Ne\leq \Cr_m(K)$}.
To emphasize the minimal set of representatives $\{\mathfrak u_\a \in \Ue\}$ of double cosets $\{N\mathfrak u N\}_{\mathfrak u \in \Ue}$ we use the following construction. Using the multiplication of $G$ by the elements of the group $N$ on left-right sides we construct a subgroup $\Ne \leq \Cr_m(K)$ of the Cremona group $\Cr_m(K)$ which acts {\it  partially } (see, section 1) on $\Ue \approx A^m_K$ and we get the following

\bigskip

{\bf Theorem 1.} {\it Elements $\mathfrak u_1, \mathfrak u_2 \in \Ue$ belong to the same double coset $N \mathfrak u  N$ if and only if they are in the same $\Ne$-orbit.}

\bigskip

Hence we  have

{\bf Corollary 1.} {\it There is one-to-one correspondence between the set of  $G$-orbits  of the pairs of maximal tori of $G$ and the $\Ne$-orbits of the subgroup $\Ne\leq \Cr_m(N)$ with respect to the partial action on $A_K^m$. }

\bigskip

The definition of the group $\Ne$ depends on the group $W\times W$ but it is not unique.  In particular, it depends on the choices of the preimages $\dot w$ of elements of Weyl group $W= N/T$ in $N$.

\bigskip

{\it The problem of pair matrices.}  There is a reduction of the well-known ``wild'' problem of {\it the classification of the pairs of matrices} $(A, B)$  up to conjugation   by a single element of $\GL_n(\C)$, where $A, B \in \Me_n(\C)$. Namely, we may change  the description of all $\GL_n(K)$-orbits of the action
$(A, B)\rightarrow (gAg^{-1}, gBg^{-1})$  to the description of fibers $$\pi: \Me_n(\C)\times \Me_n(\C)\rightarrow \big(\Me_n(\C)\times \Me_n(\C)\big)/G$$ where $\pi$ is the quotient map and   $\big(\Me_n(\C)\times \Me_n(\C)\big)/G$ is the corresponding algebraic factor (see, \cite{VP}, 9.5).

We also may formulate the following ``subproblem'': to classify $\GL_n(\C)$-orbits on the set $C_A\times C_B$ where $C_A, C_B\subset \Me_n(\C)$ are the $\GL_n(K)$-orbits of  $A, B$ for given $A, B$.  We also may give such a classification  ``up to an isomorphism of centralizers''. Say, if we take $A, B$ to be regular semisimple unimodular matrices the classification of such pairs  ``up to  an isomorphism of centralizers'' is exactly the classification of the pairs of maximal tori of $\SL_n(\C)$.

\bigskip

{\it Case $G = \SL_2(\C)$.} Here we calculate the system of representatives  of maximal  tori  for the simplest case $G = \SL_2(\C)$ (see, Corollary \ref{hhhhjhhjjs1}). Namely, let
$$g_\a:= \begin{pmatrix} 1&\a\cr 1& 1+\a\cr\end{pmatrix} \in \Ue \subset \SL_2(\C)$$
and let

$$\mathcal K : = \{\a = a+ bi \in \C\,\,\,\mid\,\,\, a \geq -\frac{1}{2}\} \setminus \{\a  = -\frac{1}{2} + bi \in \C\,\,\mid\,\,\,b <0\}.$$

\bigskip

{\bf Theorem 2.}
{\it The set of pairs
$$ (T, T) \cup ( g_\a T g_\a^{-1}, T)_{\a \in \mathcal K}$$
is a minimal set of representatives of  the orbits of the pairs of tori of $G\times G$ under conjugation by elements of $G$}.

\bigskip

Note, that $\mathcal K$ is a subset of $\C$ which is homeomorphic to   $\C$ in the standard topology (indeed, the map $\varepsilon: \mathcal K \rightarrow \C$ which is given by the formula $\epsilon(z) = (z+\frac{1}{2})^2$ defines the corresponding homeomorphism).

At the end of the paper we also consider the description of $\SL_2(\C)$-orbits of pairs $(s, t)$ of non-central semisimple elements of $\SL_2(K)$.

\bigskip

{\bf Notation and terminology.}

Here

$K$ is an algebraically closed field;

$\C$ is a complex number field;

$G$ is a simple algebraic group over $K$ (here we identify $G$ with the group of $K$ points $G(K)$) that corresponds to the root system $R$ of the rank $r$,  $R = R^+\cup R^-$ is the fixed decomposition into the union of positive and negative roots;

we consider  only  Zariski topology on $G$;

$T$ is a fixed maximal torus of $G$, $B\leq G$ is a fixed Borel subgroup that contains $T$,  $B^- = w_0 B w_0$ is the opposite Borel subgroup (here $w_0$ is the longest element of the Weyl group);

$U = R_u(B),\,\,U^- = R_u(B^-)$ are the unipotent radicals of $B, B^-$;

$N = N_G(T)$ is the normalizer of $T$;

$N/T = W$ is the Weyl group of $G$;

for $w \in W$ by $\dot w\in N$ we denote any preimage of $w$.

\section{Partial action of the Cremona subgroups}

The partial action of the group is used in different cases (see, for instance, \cite{A}, \cite{E}). Here we give the definition of the partial action in a little bit different form.

\begin{definition}
\label{defhhhfkkhgfs1}
 Let $\Gamma$ be a group and let $X$ be a set. We say that a partial action of $\Gamma$ on $X$  is defined  if for every $x \in X$
 a subset $\Gamma(x)\subset \Gamma$ is fixed and the following conditions hold:

 i. for every $\sigma \in \Gamma(x)$ an element  $\sigma(x)\in X$ is defined;

 ii. the identity $e \in \Gamma$ belongs to every $\Gamma(x)$ and $e(x) = x$;

 iii. if $\tau \in \Gamma(\sigma(x))$ then $\tau\sigma \in \Gamma(x)$ and
  $\tau(\sigma(x)) =\tau\sigma(x)$;

 iv. $\sigma^{-1}\in \Gamma(\sigma(x))$.
  \end{definition}

\bigskip

 {\it Orbits of partial action.}  For  partial actions we may define {\it orbits} of such actions. Namely, we say that  $x, y \in X$ belong to the same $\Gamma$-orbit if and only if one can find an element $\sigma \in \Gamma(x)$ such that $\sigma(x) = y$.  Obviously, the conditions  $i.- iv.$ of the Definition \ref{defhhhfkkhgfs1} can guarantee that the $\Gamma$-orbits are  classes of equivalence under the equivalence $$x\sim_\Gamma y\Leftrightarrow \sigma(x) = y\,\,\text{for some}\,\,\sigma \in \Gamma (x).$$

  {\it  The Cremona group action on the affine space}. Here we consider the Cremona group $\Cr_n(K)$ as the group of automorphisms of
  the field $K(x_1, \ldots, x_n)$ (see, for instance, \cite{S}).  Let
   $$A_K^n = \{a = (\a_1, \ldots, \a_n)\,\,\mid\,\,\a_i \in K\}$$
   be the $n$-dimensional  affine space. Every element $\sigma \in \Cr_n(K)$ is presented by the sequence of rational functions
  $$\sigma = \Big(\frac{\phi_1}{\psi_1}, \ldots, \frac{\phi_n}{\psi_n}\Big),$$
  where $\phi_i, \psi_i \in K[x_1, \ldots, x_n], \,\psi_i \ne 0,\, (\phi_i, \psi_i) = 1$, and we may define
  $$\sigma(a)\stackrel{def} =  \Big(\frac{\phi_1(a)}{\psi_1(a)}, \ldots, \frac{\phi_n(a)}{\psi_n(a)}\Big)$$
  for every point $a \in A_K^n$ such that $$a \in U_\sigma := \{a^\prime  \in A_K^n\,\,\,\mid\,\,\,\psi_i(a^\prime) \ne 0\,\,\,\text{for every}\,\,\,i\}.$$

  The set $U_\sigma$ is an open subset of $A_K^n$ where the rational map $\sigma$ is regular.

  \bigskip
 Let
 $$\tau = \Big(\frac{\mu_1}{\nu_1}, \ldots, \frac{\mu_n}{\nu_n}\Big)\in \Cr_n(K).$$
 If $\sigma(a) \in U_\tau := \{a^\prime  \in A_K^n\,\,\,\mid\,\,\,\nu_i(a^\prime) \ne 0\,\,\,\text{for every}\,\,\,i\}$ then the image $\tau(\sigma(a))$ is defined and equal to $(\tau \sigma)(a)$ where
 $$\tau \sigma =  \Big(\frac{\mu_1\big(\frac{\phi_1}{\psi_1}, \ldots, \frac{\phi_n}{\psi_n}\big)}{\nu_1\big(\frac{\phi_1}{\psi_1}, \frac{\phi_2}{\psi_2}, \ldots, \frac{\phi_n}{\psi_n}\big)}, \ldots, \frac{\mu_n\big(\frac{\phi_1}{\psi_1}, \frac{\phi_2}{\psi_2}, \ldots, \frac{\phi_n}{\psi_n}\big)}{\nu_n\big(\frac{\phi_1}{\psi_1}, \frac{\phi_2}{\psi_2}, \ldots, \frac{\phi_n}{\psi_n}\big)}\Big).$$
 Now let $\Gamma\leq\Cr_n(K)$. For every $x \in A_K^n$ we put $\Gamma(x):= \{\sigma \in \Gamma\,\,\mid\,\, x \in U_\sigma\}$. Obviously,
 conditions i.-iii. hold for $\Gamma$. However,  the condition iv. does not necessarily  hold. For instance, let $n = 2$ and $\Gamma = \langle \sigma = (x_1, \frac{x_1 - 1}{x_2})\rangle $. Then $\sigma^{-1} =\sigma$ and $\sigma^{-1}\notin \Gamma(\sigma ((1, 1)))$.
 Below we will consider some subgroups of Cremona  group $\Cr_n(K)$ which act partially on the  affine space $A_K^n$.

 \section{The decomposition $G = NU^-UN$}
We have the decomposition
\begin{equation}
\label{equaffffs1}
G = N U^- U N.
\end{equation}
Indeed,  for every Bruhat cell $B\dot wB$ we have
$B\dot w B = U_w \dot w T U$
where $U_w$ is the product (in any fixed order) of the root subgrooups $X_\alpha$
such that  $\alpha \in R^+$ and
 $w^{-1}(\alpha)\in R^-$. Hence
 $\dot w^{-1}U_w \dot w \leq U^-$  and we have$$B\dot w B \leq \dot w \dot w^{-1} U_w \dot w TU = w\underbrace{(\dot w^{-1} U_w \dot w)}_{\leq U^-} UT \leq \dot w U^- UT\leq \dot w T  U^-UN \leq NU^-UN.$$

 Put
\begin{equation}
\label{equa1d11}
\Ue = U^-U,\,\, \Ue_G = U^-TU.
\end{equation}
Then $\Ue_G$  is the Big Gauss cell of $G$ which corresponds to the Borel subgroup $B$ and $\Ue\approx A_K^m$ where $m = \dim G - \dim T = \mid R\mid$.

\subsection{The equations which define  $\bf \Ue= U^-U$}

$\,$

For a dominant   weight $\lambda: T\rightarrow K^*$ there is a regular function $\delta_\lambda$ on $G$ such that the restriction of $\delta_\lambda$ on $T$ coincides with $\lambda$ (see, for instance, \cite{EG}). We recall the construction of $\delta_\lambda$. Let
$V_\lambda$ be the irreducible $G$-module with the highest weight $\lambda$. Further, let $\mathfrak{B} = \{e_1, \ldots, e_d\}$ be the basis which consists of weight vectors of $T$ where $e_1$ corresponds to $\lambda$,  and let $\rho_\lambda: G\rightarrow \GL_n(K)$ be the matrix representation which corresponds to the basis $\mathfrak{B}$. The regular function $\delta_\lambda: G\rightarrow K$ which is defined by the formula
$\delta_\lambda(g) := g_{11}$, where $g_{11}$ is  $(1,1)$-entry of the matrix $\rho_\lambda(g)$,  satisfies the following condition:
$$ \delta_\lambda(vtu) = \lambda(t)\,\,\,\text{ for every}\,\, t \in T\,\,\text{ and for every}\,\, v \in U^-, u \in U.$$

Further, let $\delta_{\lambda_1}, \ldots, \delta_{\lambda_r}$ be regular functions on $G$  that correspond to the  fundamental weights $\lambda_1, \ldots, \lambda_r$. Then the Big Gauss cell $\Ue_G$ is defined by the inequalities
\begin{equation}
\label{equaddddfs1s11}
g \in \Ue_G \Leftrightarrow \delta_{\lambda_i} (g) \ne 0\,\,\,\text{for every}\,\,\,i  = 1, \ldots, r.
\end{equation}
The closed subset $\Ue\subset G$ is defined by equations
\begin{equation}
\label{equaddddfs1s11}
g \in \Ue \Leftrightarrow  \delta_{\lambda_i} (g)  = 1\,\,\,\text{for every}\,\,\,i  = 1, \ldots, r.
\end{equation}
\begin{remark}
Note, that in the case $G = \SL_{r+1}(K)$ the value  $\delta_{\lambda_i}(g)$ is  the  principal $i^{th}$-minor of the matrix $g \in \SL_{r+1}(K)$.
\end{remark}

\subsection{The rational map $\delta^*: G \rightarrow T$}

 $\,$

 Consider the regular  map
 $$\delta: G\rightarrow A^r_K$$
 which is defined by the formula
 $$\delta(g) = (\delta_{\lambda_1}(g), \ldots, \delta_{\lambda_r}(g)) .$$
 The definition of $\delta$ imples that
 \begin{equation}
 \label{equahhhhhjhhhhjhhhjks}
 \delta(vtu) = \delta (t)\,\,\,\,\text{for every}\,\,\,v \in U^-, t \in T,  u \in U
 \end{equation}

 The restriction of the map $\delta$ on $T$ gives us an isomorphism
 $$T\stackrel{\delta}{\approx}(A_K^r)^*:= \{(a_1, \ldots, a_r)\,\,\mid\,\, a_i \ne 0\,\,\,\text{for every}\,\,\,i\}\approx K^{* r}.$$

 We will identify the group $G$ with the closed subgroup of $\SL_n(K)$ for some $n$ where the fixed torus $T$ is a subgroup of the group of diagonal matrices of $\SL_n(K)$. Thus, if $t \in T$ then
 $$t = \di(t_1, t_2, \ldots, t_n)\,\,\,\text{for some}\,\,\,t_i \in K.$$
 Let $$\kappa: (A^r_K)^* \rightarrow T$$ be the regular map such that
 $$\kappa\circ \delta (t) = t\,\,\,\text{for every}\,\,\ t \in T.$$
 Since the restriction $\delta_{\mid\, T}: T\rightarrow (A^r)^*$ is a rational homomorphism of  the torus $T$ (recall, that $\delta(t) = (\delta_{\lambda_1}(t), \ldots, \delta_{\lambda_r}(t)) = (\lambda_1(t), \ldots, \lambda_r(t))$ for every $t \in T$), then $\kappa$ is also a rational homomorphism of the torus $ (A^r)^*$. Hence
 \begin{equation}
 \label{equahhfffhs1}
 \kappa((a_1, \ldots, a_r)) = (\prod_{i=1}^r a_i^{z_{1i}}, \ldots, \prod_{i =1}^r a_i^{z_{ni}})
  \end{equation}
  for some $z_{ij}\in \Z$. Since $\det \di (t_1, \ldots, t_n) = 1$  we have
  \begin{equation}
  \label{equahhhjjhjhhhhhh21hs1}
  \sum_{j =1}^n z_{j i} = 0\,\,\text{for every}\,\, i  = 1, \ldots, r.
  \end{equation}

  Thus  we may consider the map $\kappa$ as a rational map
 $$\kappa: A^r_K\rightarrow A^n_K$$
  which is defined by  formula \ref{equahhfffhs1}. Then the map $\kappa$  is regular at the point $(a_1, \ldots, a_r)$  if and only if  $(a_1, \ldots, a_r) \in (A_K^r)^*$ (see, \ref{equahhhjjhjhhhhhh21hs1}).

 \bigskip

 {\bf Example.}  Let $G = \SL_{r+1}(K)$. Let $t = \di (t_1, \ldots, t_r, \frac{1}{t_1t_2\ldots t_r}) \in T$. Then $\delta(t) = (t_1, t_1t_2, \ldots, t_1t_2 \ldots t_r)$ and therefore $\kappa((a_1, \ldots a_r)) = (a_1, \frac{a_2}{a_1}, \ldots, \frac{a_r}{a_{r-1}}, \frac{1}{a_r})$.

\bigskip

Now we define the rational map
$$\delta^* = \kappa\circ\delta: G\rightarrow A^n_K.$$
This map is regular only at points of the Big Gauss cell $\Ue_G$ and the image $\Imm \delta^*$ is isomorphic to $T$. We will identify $\Imm \delta^*$ with the torus $T$.
It follows directly from the definition
\begin{equation}
\label{equahhjhhjffs1}
\delta^*(t) = t\,\,\,\,\text{for every}\,\,\, t \in T.
\end{equation}
From \ref{equahhhhhjhhhhjhhhjks} we get
\begin{equation}
\label{equajhjhhhhjfs1}
\delta^*( t_1vt ut_2) = t_1 t t_2\,\,\,\text{for every}\,\,\,\,t_1, t_2, t \in T,\,\,v \in U^-,\,\,\,u \in U.
\end{equation}

\subsection{The rational map ${\bf w}_{\dot w_1, \,\dot w_2}:  \Ue\rightarrow \Ue$}

$\,$

Let $\omega: G\rightarrow G$ be an isomorphism of the affine variety $G$. Recall, that the closed subset $\Ue = U^-U$ we consider also as the affine variety $A_K^m \approx A^{\frac{m}{2}}_K \times A^{\frac{m}{2}}_K$.

Consider the rational map
${\bf w}_\omega: G\rightarrow G$
which is given by the formula $${\bf w}_\omega(g) = \big(\delta^*(\omega(g))\big)^{-1} \omega(g).$$
Let  map  ${\bf w}_\omega$ be regular at a point  $g \in \Ue$. Then  the element $\omega(g)$ belongs to the Big Gauss cell $\Ue_G \subset G$ and therefore $\omega(g) = v t u$ for some $v \in U^-, t \in T, u\in U$. Hence $\delta^* (\omega(g)) = t$ (see, \ref{equajhjhhhhjfs1}) and therefore
$${\bf w}_\omega (g)  = t^{-1} vt u = (t^{-1}vt) t^{-1}t u  =  (t^{-1}vt)  u\in \Ue,$$
Let $$\Ue_\omega:= \{ u \in \Ue\,\,\,\mid\,\,\, {\bf w}_\omega\,\,\,\text{is regular at the point}\,\,u\}.$$

 Consider the restriction of ${\bf w}_\omega$ on the closed subset $\Ue$.  If $\Ue_\omega \ne\emptyset$ then we may and we will consider ${\bf w}_\omega$ as  a rational map  ${\bf w}_\omega: \Ue\rightarrow \Ue$. Since $\omega$ is an isomorphism of varieties we have from the definition of ${\bf w}_\omega$ the equivalence
$$g \in \Ue_\omega \Leftrightarrow \omega(g) \in \Ue_G.$$
Hence
\begin{equation}
\label{equahhjhhhjkjjjkjjjks1}
\Ue_\omega = \Ue \cap \omega^{-1} (\Ue_G).
\end{equation}

\bigskip

Now we consider the map $\omega: \Ue\rightarrow \Ue$ which is defined by left-right multiplication of elements from subgroup $N$. Namely,
let $w_1, w_2\in W$ and let $\dot w_1, \dot w_2 \in N$ be fixed preimages. Further,  let $\omega: G\rightarrow G$ be the isomorphism (as a variety)  which is given by the formula $\omega(g) = \dot w_1g \dot w_2$. Then we put $${\bf w}_{\dot w_1,\, \dot w_2} :={\bf w}_\omega,\,\,\,\Ue_{w_1,\, w_2}:= \Ue_\omega.$$
The set $\Ue_{w_1, \,w_2} = \Ue \cap \dot w_1^{-1}\Ue_G \dot w_2^{-1}$ is an intersection of a closed  and an open subsets of $G$ and therefore it is an open subset in $\Ue$. Since $\Ue_G = U^-TU$ the set $\dot w_1^{-1}\Ue_G \dot w_2^{-1}$ does not depend on the choice  of preimages $\dot w_1, \dot w_2$ of $w_1, w_2$. Hence the set $\Ue_{w_1, \, w_2}$ does not also depend on the choice of preimages $\dot w_1, \dot w_2$.
\begin{lemma}
\label{lemff21ffdfffs1}
$$\Ue_{w_1, \,w_2}\ne \emptyset.$$
\end{lemma}
 \begin{proof}
 Since $\Ue_G\ne \emptyset $ is an open subset in $G$ we have $\dot w_1 \Ue_G \dot w_2 \cap \Ue_G\ne \emptyset$. Let $$\dot w_1 vtu \dot w_2 = v^\prime t^\prime u^\prime \in \dot w_1 \Ue_G \dot w_2 \cap \Ue_G$$ where $v, v^\prime \in U^-, t, t^\prime \in T, u, u^\prime \in U$. Let $s = \dot w_1 t^{-1} \dot w_1^{-1}$. Then $s \in T$ and
 $$s \dot w_1 vtu \dot w_2 = (\dot w_1 t^{-1} \dot w_1^{-1})\dot w_1 vtu \dot w_2  = \dot w_1\underbrace{\underbrace{( t^{-1}vt)}_{\in U^-} u}_{\in \Ue}\dot w_2 =$$
 $$ = sv^\prime t^\prime u^\prime =  \underbrace{(sv^\prime s^{-1})}_{\in U^-}\underbrace{(st^\prime)}_{\in T} u^\prime \in \Ue_G\Rightarrow
 \Ue \cap \dot w_1^{-1}\Ue_G \dot w_2^{-1}\ne\emptyset.$$

 \end{proof}
Thus we have the rational map
$${\bf w}_{\dot w_1, \, \dot w_2}: \Ue\rightarrow \Ue$$
such that the set of points $\mathfrak u \in \Ue$ where ${\bf w}_{\dot w_1, \,\dot w_2}$ is regular coincides with a  non-empty open subset $\Ue_{w_1, \, w_2} = \Ue \cap \dot w_1^{-1}\Ue_G \dot w_2^{-1}$.
\begin{prop}
\label{lemhhjhhjfs1}
The map $$ {\bf w}_{\dot w_1, \,\dot w_2}: \Ue_{w_1, w_2}\rightarrow {\bf w}_{\dot w_1,\, \dot w_2}(\Ue_{w_1,\, w_2})$$
is an isomorphism of  open subsets of $\Ue$ and
${\bf w}_{\dot w_1,\, \dot w_2}(\Ue_{w_1,\, w_2}) = \Ue_{w^{-1}_1, w_2^{-1}}$.
\end{prop}
\begin{proof}
Let $\mathfrak u\in \Ue_{w_1, w_2} $. Further, let
\begin{equation}
\label{equahhhfhs1}
\dot w_1\mathfrak u \dot w_2 = v t u
\end{equation}
 for some $v\in U^-, u \in U, t\in T$. Then
$$\mathfrak u^\prime = {\bf w}_{\dot w_1, \dot w_2}(\mathfrak u) = t^{-1}v t u.$$
Further,
$${\bf w}_{\dot w_1^{-1}, \,\dot w_2^{-1}}(\mathfrak u^\prime) = $$$$= \big(\delta^*(\dot w_1^{-1}\mathfrak u^\prime\dot w_2^{-1})\big)^{-1}\dot w_1^{-1}\mathfrak u^\prime \dot w_2^{-1} =\big(\delta^*(\dot w_1^{-1}\mathfrak u^\prime\dot w_2^{-1})\big)^{-1} \underbrace{(\dot w_1^{-1}t^{-1}\dot w_1)}_{:= s \in T} \underbrace{(\dot w_1^{-1}v t u\dot w_2^{-1})}_{= \mathfrak u;\, (\text{see}\,\,\,\ref{equahhhfhs1})} =$$$$= \big(\underbrace{\delta^*(s\mathfrak u)}_{= s}\big)^{-1}s\mathfrak u  = \mathfrak u.$$
Hence ${\bf w}_{\dot w_1^{-1}, \,\dot w_2^{-1}}(\mathfrak u^\prime) = \mathfrak u$ and therefore
\begin{equation}
\label{equahhhklkkhs1}
\text{the map}\,\,\,{\bf w}_{\dot w_1^{-1},\, \dot w_2^{-1}}\circ {\bf w}_{\dot w_1, \, \dot w_2}: \Ue_{w_1, \,w_2}\rightarrow \Ue_{w_1,\, w_2}\,\,\,\text{is the identity}.
\end{equation}
Since \ref{equahhhklkkhs1} holds for every pair of $\dot w_1, \dot w_2$ we have
\begin{equation}
\label{equahjhjkkkjhjkjhs1}
{\bf w}_{\dot w_1, \,\dot w_2}\circ {\bf w}_{\dot w_1^{-1}, \,\dot w_2^{-1}}:  \Ue_{w_1^{-1}, \,w_2^{-1}}\rightarrow \Ue_{w_1^{-1}, \, w_2^{-1}} \,\,\,\,\,\,\text{is the identity}.
\end{equation}

 Since the map ${\bf w}_{\dot w_1^{-1}, \,\dot w_2^{-1}}$ is regular at  every point $\mathfrak u^\prime = {\bf w}_{\dot w_1, \dot w_2}(\mathfrak u) $ we have the inclusion
\begin{equation}
\label{equahhhgfs1}
{\bf w}_{\dot w_1, \, \dot w_2}(\Ue_{w_1, \,w_2}) \subset \Ue_{w_1^{-1}, \,w_2^{-1}}.
\end{equation}
Then the  inclusion \ref{equahhhgfs1} holds if we change $w_i $ for $w_i^{-1}$
\begin{equation}
\label{equajffklfffhjs1}
{\bf w}_{\dot w_1^{-1}, \,\dot w_2^{-1}}(\Ue_{w^{-1}_1, \,w^{-1}_2}) \subset \Ue_{w_1, \, w_2}.
\end{equation}
We may consider  maps
$$\Ue_{w_1, \,w_2}\stackrel{{\bf w}_{\dot w_1, \, \dot w_2}}{\longrightarrow} \Ue_{w_1^{-1}, \,w_2^{-1}}\stackrel{{\bf w}_{\dot w_1^{-1}, \, \dot w_2^{-1}}}{\longrightarrow} \Ue_{w_1, \,w_2}\stackrel{{\bf w}_{\dot w_1, \, \dot w_2}}{\longrightarrow} \Ue_{w_1^{-1}, \,w_2^{-1}}$$
(see,  \ref{equahhhgfs1} and \ref{equajffklfffhjs1}). Thus, \ref{equahhhklkkhs1} and \ref{equajffklfffhjs1}  imply that
$$ {\bf w}_{\dot w_1, \,\dot w_2}: \Ue_{w_1, w_2}\rightarrow {\bf w}_{\dot w_1,\, \dot w_2}(\Ue_{w_1,\, w_2})$$
is an isomorphism of  open subsets of $\Ue$ and
${\bf w}_{\dot w_1,\, \dot w_2}(\Ue_{w_1,\, w_2}) = \Ue_{w^{-1}_1, w_2^{-1}}$.

\end{proof}

Now we get
\begin{cor}
\label{corhhjhjjhs1}
${\bf w}_{\dot w_1, \,\dot w_2} \in \Cr_m(K)$ and ${\bf w}_{\dot w_1, \, \dot w_2}^{-1} = {\bf w}_{\dot w^{-1}_1, \, \dot w_2^{-1}}$.
\end{cor}

\begin{remark}
\label{remjjjhffks1}
Note, that ${\bf w}_{\dot w^{-1}_1, \, \dot w_2^{-1}}$ is the map that corresponds to $\dot w^{-1}_1, \dot w_2^{-1}$, that is, we take here  fixed preimages $\dot w_1, \dot w_2$ of $w_1, w_2$ and then take the inverse elements of these preimeges. If we take  other preimages $\dot{ (w_1^{-1})}, \dot {(w_2^{-1})}$ we get the corresponding map ${\bf w}_{\dot{( w^{-1}_1)}, \,\dot {(w_2^{-1})}}$ which differs from ${\bf w}_{\dot w_1, \,\dot w_2}^{-1}$ on a multiplier from $T$.
\end{remark}

\section{Definition of the  group $\Ne$}

Now we define some  group $\Ne \leq Cr_m(K)$ which partially acts on $\Ue\approx A^m$.

\begin{lemma}
\label{lemffffdfffs1}
Let $\mathfrak u \in \Ue$ and  let $t, s \in T$. Then $$t\mathfrak u s \in \Ue \Leftrightarrow s = t^{-1}.$$
\end{lemma}
\begin{proof}
Let $\mathfrak u = vu$ where $v\in U^-, u \in U$. Then
$$t\mathfrak u s = \underbrace{(t v t^{-1})}_{\in U^-} \underbrace{t s}_{\in  T}\underbrace{(s^{-1} u s)}_{\in U}\in \Ue \Leftrightarrow ts = 1$$
(recall, that the Gauss decomposition $g = vtu$ of every element $g \in \Ue_G$ is unique).
\end{proof}

 For every $s \in T$ we define the transformation of $$\mathfrak t_s: \Ue\rightarrow \Ue$$
that is defined by the formula
$$\mathfrak t_s (\mathfrak u) = s \mathfrak u s^{-1}.$$
Let
$$\mathcal T := \langle \mathfrak t_s\,\,\,\mid\,\,s \in T\rangle.$$
Then we may consider $\mathcal T$ as a subgroup of $\Cr_m(K)$. Further, we assume that we have the set of fixed preimages $\{\dot w\}_{w\in W}$. Put
$$\Ne  = \langle \mathcal T, {\bf w}_{\dot w_1, \dot w_2}\,\,\,\mid\,\,\,(w_1, w_2) \in W\times W\rangle \leq \Cr_m(K).$$
Note, that $\Ne \leq \Cr_m(K)$ is a group that  depends on the choice of preimages $\{\dot w\}_{w\in W}$. Moreover, the inverse elements
${\bf w}_{\dot w_1, \dot w_2}^{-1}$ do not necessary coincide with ${\bf w}_{\dot {(w_1^{-1})}, \dot{( w_2^{-1})}}$  (see, Remark \ref{remjjjhffks1}). More precisely
\begin{lemma}
\label{lemhhjfjjhfjhjjhfs1}
Let $\{\dot w\}_{w\in W}$ be a fixed set of preimages in $N$ elements of the Weyl group. Further, let
$$\omega_1 = w_1^\prime w_1, \omega_2 = w_2 w_2^\prime$$
Then there exists an element $\mathfrak t_s \in \mathcal T$ such that
$${\bf w}_{\dot w_1^\prime, \, \dot w_2^\prime }{\bf w}_{\dot w_1, \, \dot w_2} = \mathfrak t_s {\bf w}_{\dot \omega_1,\,  \dot \omega_2}.$$
\end{lemma}
\begin{proof}
We have
\begin{equation}
\label{equajhjhhfhjhhfs1}
t_1 \dot \omega_1 =  \dot w_1^\prime \dot w_1, \,\,\, \dot \omega_2t_2 = \dot w_2 \dot w_2^\prime \,\,\,\text{for some}\,\,\, t_1, t_2 \in T.
\end{equation}
Let $$\mathfrak u \in \Ue_{w_1, w_2}\cap {\bf w}^{-1}_{\dot w_1, \dot w_2}(\Ue_{w^\prime_1, w^\prime_2}\cap \Ue_{w_1^{-1}, w_2^{-1}})$$ (note, that the intresections  of non-empty open subsets of $\Ue$ are non-empty open subsets). Hence the map ${\bf w}_{\dot w_1^\prime, \, \dot w_2^\prime }$ is regular at the point ${\bf w}_{\dot w_1,\,  \dot w_2}(\mathfrak u) $ and we have
$${\bf w}_{\dot w_1^\prime, \,\dot w_2^\prime }{\bf w}_{\dot w_1,\, \dot w_2}(\mathfrak u) = {\bf w}_{\dot w_1^\prime, \, \dot w_2^\prime} \Big(\underbrace{\big (\delta^* (\dot w_1 \mathfrak u \dot w_2)\big )^{-1} }_{t \in T} \dot w_1 \mathfrak u \dot w_2\Big) =$$
$$ =\big (\delta^*((\dot w_1^\prime t \dot w_1^{\prime -1})\dot w_1^\prime (\dot w_1 \mathfrak u \dot w_2) \dot w_2^\prime)\big )^{-1}  (\dot w_1^\prime t \dot w_1^{\prime -1})\dot w_1^\prime (\dot w_1 \mathfrak u \dot w_2) \dot w_2^\prime  \stackrel{\ref{equajhjhhhhjfs1}}{= }$$
$$= \big (\delta^* (\dot w_1^\prime (\dot w_1 \mathfrak u \dot w_2) \dot w_2^\prime)\big )^{-1}  \dot w_1^\prime (\dot w_1 \mathfrak u \dot w_2) \dot w_2^\prime \stackrel{\ref{equajhjhhfhjhhfs1}}{=} \big (\delta^* (t_1 \dot \omega_1 \mathfrak u \dot \omega_2 t_2)\big )^{-1}  t_1 \dot \omega_1 \mathfrak u \dot \omega_2 t_2 = $$$$ \stackrel{\ref{equajhjhhhhjfs1}}{= }t_2^{-1}\big (\delta^* (\dot \omega_1 \mathfrak u \dot \omega_2 )\big )^{-1}   \dot \omega_1 \mathfrak u \dot \omega_2 t_2 = \mathfrak t_s {\bf w}_{\dot \omega_1, \dot \omega_2}\,\,\, \text{where}\,\,\,s  = t_2^{-1}.$$

\end{proof}
\begin{lemma}
\label{lemjjhjhhhfjjjfs1}
The subgroup $\mathcal T$ is normal in $\Ne$.
\end{lemma}
\begin{proof}
Let $\mathfrak t_s\in \mathcal T,\,\,{\bf w}_{\dot w_1,\, \dot w_2}\in \Ne$ and let $\mathfrak u \in \Ue_{w_1, w_2}$. Then
$$\mathfrak t_s{\bf w}_{\dot w_1,\, \dot w_2}  (\mathfrak u) = \underbrace{s\Big(\big(\delta^*(\dot w_1 \mathfrak u \dot w_2)\big)^{-1}}_{ := t \in T}(\dot w_1 \mathfrak u  \dot w_2)\Big) s^{-1}  = t (\dot w_1 \mathfrak u  \dot w_2) s^{-1}.$$
Then
$${\bf w}^{-1}_{\dot w_1,\, \dot w_2}\mathfrak t_s{\bf w}_{\dot w_1,\, \dot w_2}  (\mathfrak u) \stackrel{\text{Cor.}\, \ref{corhhjhjjhs1}}{=}
{\bf w}_{\dot w_1^{-1},\, \dot w_2^{-1}} \mathfrak t_s{\bf w}_{\dot w_1,\, \dot w_2}(\mathfrak u) =  $$$$=\underbrace{\Big (\delta^*(\dot w_1^{-1} t (\dot w_1
\mathfrak u  \dot w_2) s^{-1} \dot w_2^{-1})\Big)^{-1} (\dot w_1^{-1} t (\dot w_1}_{:=t_1 \in T} \mathfrak u  \underbrace{\dot w_2) s^{-1}\dot w_2^{-1})}_{:= t_2 \in T} = t_1 \mathfrak u t_2\stackrel{\text{Lem.}\,\ref{lemffffdfffs1}}{ =} $$$$ = t_2^{-1}\mathfrak u t_2 =
\mathfrak t_{s^\prime}(\mathfrak u)\,\,\text{where}\,\,s^\prime = t_2^{-1}. $$
\end{proof}

Now we show that the group $\Ne$ acts partially on the affine space $\Ue$. For every $\mathfrak u$ we define $\Ne(\mathfrak u)$ as the set of elements of $\Ne$ which are regular at the point $\mathfrak u$. Then conditions i. - iii. of the Definition \ref{defhhhfkkhgfs1} obviously hold.
 From Proposition \ref{lemhhjhhjfs1} and  Corollary \ref{corhhjhjjhs1} we have condition iv:
$$\mathfrak n^{-1}  \in \Ne(\mathfrak n(\mathfrak u))\,\,\,\text{for every}\,\,\,\mathfrak n \in \Ne(\mathfrak \mathfrak u).$$

\bigskip

Now we may prove Theorem 1.

\begin{proof}

The implication $\Leftarrow$ follows directly from the definition of double cosets.

Let $\mathfrak u_1, \mathfrak u_2 \in \Ue$. Then
$$\mathfrak u_1, \mathfrak u_2 \in N \mathfrak u N \stackrel{\ref{equaffffs1}}{\Leftrightarrow} n_1 \mathfrak u_1n_2 = \mathfrak u_2\,\,\, \text{for some }\,\,\, n_1, n_2 \in N.$$
Let $n_1 \mathfrak u_1n_2 = \mathfrak u_2$. Then $n_1 = t_1 \dot w_1, n_2 = \dot w_2 t_2$ for some $t_1, t_2 \in T$. Since $\mathfrak u_1, \mathfrak u_2 \in \Ue$ then $\mathfrak u_1 \in \Ue_{w_1, w_2}$.
 We have
$$n_1 \mathfrak u_1n_2 = t_1 \dot w_1 \mathfrak u_1 \dot w_2 t_2 =   \underbrace{t_1  (\delta^* (\dot w_1\mathfrak u\dot w_2))}_{:= t_1^\prime\in T}\underbrace{{\bf w}_{\dot w_1, \, \dot w_2}(\mathfrak u)}_{\in \Ue} t_2 = \mathfrak u_2 \in \Ue\stackrel{\text{Lemma} \,\ref{lemffffdfffs1}}{\Rightarrow} t^\prime_1t_2 = 1 \Rightarrow $$
$$\Rightarrow u_2 = \mathfrak t_s {\bf w}_{w_1, w_2}(\mathfrak u_1)\,\,\,\text{where}\,\,\,s = t_2^{-1}.$$
\end{proof}

\section{Example. Case $G = \SL_2(\C)$}

\subsection{Group $\Ne$}

$\,$

Let $G = \SL_2(K)$. Then
$$\Ue = \{\begin{pmatrix} 1&\alpha\cr
\beta & 1+\alpha \beta \cr \end{pmatrix}\,\,\,\mid\,\,\a, \b \in K\}.$$
Here $W = \{e, w\}$ is the group consisting of two elements --  the identity $e$ and the involution $w$. We take
 $$\dot e  = \begin{pmatrix} 1&0\cr 0&1\cr \end{pmatrix}, \,\,\dot w = \begin{pmatrix} 0&1\cr -1&0\cr \end{pmatrix}.$$
 Hence
$${\bf w}_{\dot w,\, \dot e}\Big(\begin{pmatrix} 1&\alpha\cr
\beta & 1+\alpha \beta \cr \end{pmatrix}\Big) = \begin{pmatrix} \beta^{-1}&0\cr 0&\beta\cr \end{pmatrix}\begin{pmatrix} \beta &(1+\a\b)\cr
- 1 & -\a \cr \end{pmatrix} = \begin{pmatrix} 1 &\b^{-1}(1+\a\b)\cr
- \beta & -\a\b \cr \end{pmatrix},  $$
$${\bf w}_{\dot e,\, \dot w}\Big(\begin{pmatrix} 1&\alpha\cr
\beta & 1+\alpha \beta \cr \end{pmatrix}\Big) =\begin{pmatrix}- \a^{-1}&0\cr 0&-\a\cr \end{pmatrix} \begin{pmatrix} -\a&1\cr
-(1+\a\b) & \beta \cr \end{pmatrix} =  \begin{pmatrix} 1&-\a^{-1}\cr
\a(1+\a\b) & -\a\beta \cr \end{pmatrix},  $$
$${\bf w}_{\dot w,\, \dot w}\Big(\begin{pmatrix} 1&\alpha\cr
\beta & 1+\alpha \beta \cr \end{pmatrix}\Big) =\begin{pmatrix} - (1+\a\b)^{-1}&0\cr 0&-(1+\a\b)\cr \end{pmatrix} \begin{pmatrix}
-(1+\a\b) & \beta \cr
 \a&-1\cr
 \end{pmatrix} = $$
 $$ = \begin{pmatrix}
 1& -\beta (1 +\a\b)^{-1}\cr
 -\a(1+\a\b)&(1+\a\b)\cr
 \end{pmatrix}.$$
 It is easy to check
 \begin{equation}
 \label{equajjhjhhffhjs1}
 {\bf w}_{\dot e,\, \dot w}{\bf w}_{\dot w,\, \dot e} ={\bf w}_{\dot w,\, \dot e} {\bf w}_{\dot e,\, \dot w} = {\bf w}_{\dot w,\, \dot w}\,\,\,\text{and}\,\,\,{\bf w}_{\dot e,\, \dot w}^2 = {\bf w}_{\dot w,\, \dot e}^2 = {\bf w}_{\dot e,\, \dot e} = {\bf e}
 \end{equation}
(here ${\bf e}$ is the identity of $\Cr_2(K)$). Put $${\bf w}_l:= {\bf w}_{\dot w,\, \dot e},\,\,\,{\bf w}_r:= {\bf w}_{\dot e,\, \dot w},\,\,\,{\bf w}_{d}:= {\bf w}_{\dot w,\, \dot w}.$$
Thus, ${\bf w}_l,{\bf w}_r:= {\bf w}_{\dot e,\, \dot w},{\bf w}_{d} \in \Cr_2(K)$ are birational transformations of the affine plane $$A_K^2 = \{(\a, \b)\,\,\,\mid\,\,\a, \b\in K\}.$$
Namely,
$$
{\bf w}_l((\a, \b)) = (\b^{-1}(1+\a\b), - \beta),\,\,\, {\bf w}_r((\a, \b)) = (-\a^{-1}, \a(1+\a\b)),$$
$$
{\bf w}_{d}((\a, \b)) = ( -\beta (1 +\a\b)^{-1}, -\a(1+\a\b)).$$
The element $\mathfrak t_s \in \mathcal T\leq \Cr_2(K)$ acts on $A_K^2$ according to the following formula
$$\mathfrak t_s((\a, \b)) = (s^2 \a, s^{-2}\b).$$
Also
$${\bf w}_l((\a, \b))\mathfrak t_s {\bf w}_l ((\a, \b)) = {\bf w}_l(s^2\b^{-1}(1+\a\b), - s^{-2}\beta) = (s^2 \a, s^{-2}\b) =\mathfrak t_s((\a, \b)), $$
$${\bf w}_r((\a, \b))\mathfrak t_s {\bf w}_r ((\a, \b)) = {\bf w}_r( (-s^2\a^{-1}, s^{-2}\a(1+\a\b)) = (s^{-2} \a, s^2\b) = \mathfrak t^{-1}_s ((\a, \b)).$$
Hence
$$\Ne = \underbrace{\langle \mathcal T, {\bf w}_r\rangle}_{\approx D_\infty} \times \underbrace{\langle{\bf w}_l\rangle }_{\approx \Z_2}$$
(here $D_\infty$ is the infinite dihedral group and $\Z_2 = \Z/2\Z$ is the group of order $2$).
\subsection{The decomposition of $A_K^2$ }
$\,$

Let
 $$\mathcal M = \{(\alpha, \beta) \,\,\,\mid\,\,\,\a\ne 0, \b \ne 0, \a\b\ne -1\}$$
 \begin{lemma}
 \label{lemhjhhfhhhjs1}
  Every element  $g \in \Ne$ stabilizes the open set $\mathcal M$.
 \end{lemma}
 \begin{proof}
 We have to check:  if $(\alpha, \beta) \in \mathcal M$ then $g((\alpha, \beta)) \in \mathcal M$ if $g = \t_s, \w_l, \w_r$. It is obvious for $\t_s$.
 Further,
 $$\w_l((\alpha, \beta))  = ( \underbrace{\b^{-1}(1+\a\b)}_{ = \a^\prime \ne 0 } ,  \underbrace{-\b}_{ = \b^\prime \ne 0}), {\bf w}_r((\a, \b)) = (\underbrace{-\a^{-1}}_{=\a^\prime \ne 0 }, \underbrace{\a(1+\a\b)}_{=\b^\prime \ne 0}).$$
 In both cases $\a^\prime \b^\prime = -(1+\a\b)$.
 Thus, $(1 +\a^\prime\b^\prime) = -\a\b \ne 0$
 and therefore $$\w_l((\alpha, \beta)), {\bf w}_r((\a, \b))  \in\mathcal M.$$
 \end{proof}
 Let
 $$\mathcal M_{0,1} = \{(0, \beta) \,\,\,\mid\,\,\, \b \ne 0\}, \mathcal M_{1,0} = \{(\a, 0) \,\,\,\mid\,\,\, \a \ne 0\},\,\,\,\mathcal M_{-1} = \{(\alpha, \beta) \,\,\,\mid\,\,\,\a \b = -1\}.$$

 From the definition of $\w_l, \w_r, \w_d$
 we have the following formulas
\begin{equation}
\label{equahhhhhhhhjjjjjjjhhhhhs1}
 \w_l(\mathcal M_{0,1}) = \{(\b^{-1}, -\beta) \,\,\,\mid\,\,\, \b \ne 0\} = \mathcal M_{-1},
 \end{equation}
 $$\w_l(\mathcal M_{-1}) = \{(\b^{-1}(1 +\a\b), -\beta) = (0, -\beta)\,\,\,\mid\,\,\, \b \ne 0\} = \mathcal M_{0,1},
 $$
  $$\w_r(\mathcal M_{1,0}) = \{(-\a^{-1}, \a) \,\,\,\mid\,\,\, \a \ne 0\} = \mathcal M_{-1},\,\,\w_r(\mathcal M_{-1}) = \{(-\a^{-1}, 0)\,\,\,\mid\,\,\, \a \ne 0\} = \mathcal M_{1,0},$$
  $$\w_d(\mathcal M_{1,0})  = \mathcal M_{0,1},\,\,\w_d(\mathcal M_{0,1}) = \mathcal M_{1,0}.$$

 \begin{lemma}
 \label{lemhhjhhjhhjfs1}
 The set $$\tilde{\mathcal M} = \mathcal M_{0,1}  \cup \mathcal M_{0,1}\cup\mathcal M_{-1}$$ is just one $\Ne$-orbit. The point $(0, 1)$ can be taken as a representative of this orbit.
 \end{lemma}
 \begin{proof}
 Since $K$ is an  algebraically closed field the sets  $ \mathcal M_{0,1},  \mathcal M_{0,1}, \mathcal M_{-1}$ are three $\mathcal T$- orbits of points
 $(0, 1), (1,0), (1, -1)$ respectively.  But these points are in the same $\Ne$-orbit (that follows from \ref{equahhhhhhhhjjjjjjjhhhhhs1}).
 \end{proof}

\subsection{The representatives of $\Ne$-orbits on $\bf A_K^2 = \mathcal M\bigcup \tilde{\mathcal M}\bigcup \{(0,0)\}$}

$\,$

Put $$\mathcal M^1 := \{(\a, 1)\,\,\,\mid\,\,\,\a\ne 0,  -1\}\subset \mathcal M.$$
\begin{lemma}
\label{lemhhhhjjjjfs1}
If $(\a^\prime, \b^\prime) \in \mathcal M$ then there is an element $(\a, 1) \in \mathcal M^1$ which belongs to the same orbit as $(\a^\prime, \b^\prime)$.
\end{lemma}
\begin{proof}
We have $\mathfrak t_s ((\a^\prime, \b^\prime))
= (s^2\a^\prime, s^{-2}\b^\prime)$. Since $K$ is an algebraically closed field we can find $s \in K$ such that $s^{-2}\b^\prime = 1$. Hence
 in  every $\Ne$-orbit of the set $\mathcal M$ there is an element of the form $(\a^\prime, 1)$.
\end{proof}

\begin{lemma}
\label{lemjjhhhdhfdfhhjhs1}
The elements $(\a, 1)\ne (\a^\prime, 1) \in \mathcal M^1$ are in the same $\Ne$-orbit if and only if
$$\a^\prime = -1 - \a.$$
\end{lemma}
\begin{proof}
The elements $(\a, 1), (\a^\prime, 1) \in \mathcal M^1$ are in the same $\Ne$-orbit if and only if
\begin{equation}
\label{equajjhjjhfffffs1}
 \t_s {\bf w}((\a, 1)) =  (\a^\prime, 1)\,\,\, \text{for some}\,\,\,s = s(\a,\,{\bf w})\in K\,\,\text{and where}
\end{equation}
$$ {\bf w} \,\,\text{is one of the following elements}:\,\,\,{\bf e}, \,\,{\bf w}_l, \,\,{\bf w}_r,\,\,\,{\bf w}_{d}$$
(it follows from Lemmas \ref{lemhhjfjjhfjhjjhfs1},  \ref{lemjjhjhhhfjjjfs1}; recall, ${\bf e } \in \Cr_2(K)$ is the identity).
Let  ${\bf w} = {\bf e}$.  Then $$ \t_s {\bf  e} ((\a, 1)) = \t_s (((\a, 1)) \in \mathcal M^1\Leftrightarrow \t_s = {\bf  e} \Leftrightarrow \a^\prime = \a.$$
We have
\begin{equation}
\label{equahhhs1}
{\bf w}_l((\a, 1)) = ((1+\a), - 1),\,\,\,\,
{\bf w}_r((\a, 1)) = (-\a^{-1}, \a(1+\a)),
\end{equation}
$${\bf w}_{d}((\a, 1)) = ( - (1 +\a)^{-1}, -\a(1+\a)).$$
 For ${\mathbf w}  = {\bf w}_l$, or ${\mathbf w}  = {\bf w}_l$,  or  ${\mathbf w}  = {\bf w}_d$  and for the fixed $\a$ there is only one element $\t_s$  such that  $\t_s {\bf w}((\a, 1)) \in \mathcal M^1$. From \ref{equahhhs1}
\begin{equation}
\label{equahhfffhhjs1}
 s = \begin{cases} \sqrt{-1}\,\,\,\text{for the case }\,\,\,{\bf w} = {\bf w}_l,\\
$\,$\\
\sqrt{\a(1+\a))}\,\,\,\text{for the case }\,\,\,{\bf w} = {\bf w}_r,\\
$\,$\\
\sqrt{ -\a(1+\a)}\,\,\,\text{for the case }\,\,\,{\bf w} = {\bf w}_d\end{cases}
\end{equation}
(recall, that $\t_s((\a, \b)) = (s^2\a, s^{-2}\b))$ and therefore $\t_{s_1} = \t_{s_2}$ if and only if $s_1 = \pm s_2$).
 From \ref{equahhhs1}, \ref{equahhfffhhjs1}  for corresponding $s$ we get
\begin{equation}
\label{equahhff21fhhjs1}
\begin{cases} \t_s{\bf w}_l((\a, 1)) = (-(1+\a), 1)\\
\t_s{\bf w}_r((\a, 1)) = (-(1+\a), 1),\\
\t_s{\bf w}_d((\a, 1)) = (\a, 1).\end{cases}
\end{equation}
Now the statement of the Lemma follows from \ref{equajjhjjhfffffs1} and \ref{equahhff21fhhjs1}.

\end{proof}

Let $\flat : K\rightarrow K$ be the map which is given by the formula
$$\flat (\a) = -1 - \a$$
for every $\a \in K$. Then $\flat^2 $ is the identity on $K$ and there is only one element $\a$ such that $\flat(\a) = \a$, namely $\a = -\frac{1}{2}$. Then we may decompose $K = K_\flat^-\cup K_\flat$ into the union of two disjoint subsets $ K_\flat^-,\, K_\flat$  where $-\frac{1}{2}, 0 \in K_\flat$ and if $ -\frac{1}{2}\ne \a \in K_\flat$ then $\flat(\a) \in K_\flat^-$. Now let us  fix such a  decomposition  $K = K_\flat^-\cup K_\flat$.  Put
$$\Omega_K: = \{(\a, 1)\}_{\a \in K_\delta} \cup \{(0,0)\}.$$

\begin{theorem}
\label{thhjjhfffhjhfffs1}
The set $\Omega_K$ is a smallest set of the representatives of the $\Ne$-orbits on $A_K^2$, that is, for every $\Ne$-orbit there is only one its representative which is contained in $\Omega_K$.
\end{theorem}
\begin{proof}
Let $(\a, \b) \in A_K^2$ and let $O_{\a, \,\b}$ be the $\Ne$-orbit of $(\a, \b)$.  Suppose $(\a, \b) \in \mathcal M$.  Then $O_{\a, \, \b} = O_{\a^\prime,\, 1}$ where $\a^\prime \in K^+_\flat$ (see, Lemmas \ref{lemhhhhjjjjfs1}, \ref{lemjjhhhdhfdfhhjhs1}). Moreover, the definition of $K_\flat$ and Lemma \ref{lemjjhhhdhfdfhhjhs1} imply that such element $\a^\prime\in K_\flat$ is unique.
Suppose $(\a, \b) \in \tilde{\mathcal M}$ then $O_{\a, \,\b} = \tilde{\mathcal M}$ and we may take a representative $(0, 1)\in K_\flat $ of this orbit (Lemma \ref{lemhhjhhjhhjfs1}). Note, that  only  elements of the form $\mathfrak t_s$ and $\mathfrak t_s {\bf w}_d$ are regular  at the point $(0,0)$ and in these cases the point $(0, 0)$ is invariant. Hence the set $\{(0, 0)\}$ is just one $\Ne$-orbit.
\end{proof}
From Theorems 1 and  \ref{thhjjhfffhjhfffs1}  we get
\begin{cor}
\label{corhjhhhffhjjfs1}
$$\SL_2(K) = N \cup \Big( \bigcup_{\a\in K_\flat} N\begin{pmatrix}1&\a\cr
1&1+\a\cr\end{pmatrix} N\Big).$$
\end{cor}

Put
$$g_\a :=\begin{pmatrix} 1&\a\cr 1&1+\a\cr\end{pmatrix}.$$

From Corollary \ref{corhjhhhffhjjfs1} and (*) (see, Introduction) we get
\begin{cor}
\label{hhhhjhhjjs1}
The set of pairs
$$ (T, T) \cup ( g_\a T g_\a^{-1}, T)_{\a \in K_\flat}$$
is a minimal set of representatives of  the orbits of pairs of tori of $G\times G$ under conjugation by elements of $G$.
\end{cor}

\subsection{ Case $K = \C$}

$\,$

\begin{lemma}
\label{lemjjjjhjjjhhhhs1}
Let
$$\mathcal K = \{z  = a +bi \in \C\,\,\,\mid\,\,\,a \geq -\frac{1}{2}\}\setminus \{z  = -\frac{1}{2}  +bi \in \C\,\,\,\mid\,\,\, b < 0\}.$$
Then we may take $\mathcal K$ as the set $\C_\flat$.
\end{lemma}
\begin{proof}
We have $-\frac{1}{2}, 0 \in \mathcal K$. Let $ z = a+bi \in \mathcal K$. Suppose $a \ne -\frac{1}{2}$. Then
$$\flat(z) = \underbrace{(-1-a)}_{< -\frac{1}{2}} - bi \in \C \setminus \mathcal K.$$
If $a = -\frac{1}{2}$ then $b\geq 0$ and therefore
$$\flat(z)  = -\frac{1}{2} - \underbrace{b}_{<0} i \in \C \setminus \mathcal K.$$
\end{proof}
\subsection{Orbits of pairs of semisimple matrices}

$\,$

Put
$$ U^w: = \dot w U = \{\begin{pmatrix} 0&1\cr -1&0\cr \end{pmatrix}\begin{pmatrix}1& \b\cr  0 & 1\cr \end{pmatrix}\,\,| \,\,\b\in K\} = \{\begin{pmatrix}0&1\cr   -1& -\b\cr \end{pmatrix}\,\,\,|\,\,\,\b \in K\}.$$

We have  $$G = \dot w G =\dot w B \cup \dot w B\dot w^{-1}B = T\dot w U \cup\underbrace{ ( \dot w U\dot w^{-1}) }_{= U^-}  U T = T \big( U^w \cup \Ue\big)T. $$
Obviously,  for every $v \in U^w$ and $t_1, t_2\in T$  the equality $t_1 v t_2 \in U^w$ implies $t_1 = t_2$.  Thus only two $T\times T$-orbits intersect $U^w$ --  the orbit of matrix $\dot w = \begin{pmatrix}0&1\cr -1&0\cr \end{pmatrix}$ and the orbit of $ \begin{pmatrix}0&1\cr -1&1\cr \end{pmatrix}$.

 By Lemma \ref{lemffffdfffs1} for $\mathfrak u \in \Ue$ we have the inclusion  $t_1 \mathfrak u t_2 \in \Ue$ if and only if $t_2 = t_1^{-1}$. Thus, the minimal set of the representatives of double cosets of $T g T$ in the group $G$ is
 \begin{equation}
 \label{equahhjhhjhhhhfhfs1}
\underbrace{\{g_\a\}_{\a\in K}}_{\approx A_K^1}\cup \{\underbrace{\begin{pmatrix}0&1\cr -1&1\cr \end{pmatrix}}_{=: \dot w(1)}, \underbrace{\begin{pmatrix}0&1\cr -1&0\cr \end{pmatrix}}_{=\dot w}, \underbrace{\begin{pmatrix} 1&1\cr 0&1\cr \end{pmatrix}}_{=:u(1)}, \underbrace{\begin{pmatrix} 1&0\cr 0&1\cr \end{pmatrix}}_{ = \dot e}\}.
\end{equation}

\bigskip

Let $t = \begin{pmatrix} s&0\cr
0&s^{-1}\cr \end{pmatrix}$. Then
\begin{equation}
\label{eqahhhjjhffhhjs1}
g_\a t g_\a^{-1} = \begin{pmatrix} s + \a(s- s^{-1})& \a(s^{-1} - s)\cr
$\,$\cr
(1+\a)(s - s^{-1})& \,\,\,\,s^{-1} + \a(s^{-1} - s)\cr \end{pmatrix} = \begin{pmatrix} s + \a\Delta_t &- \a\Delta_t\cr
$\,$\cr
(1+\a)\Delta_t& \,\,\,\,s^{-1} - \a\Delta_t\cr \end{pmatrix}
\end{equation}
where $\Delta_t = s - s^{-1}$. Note, that the given $\Delta \in K$ corresponds only to matrices $t$ and $\bar{t} =\begin{pmatrix} -s^{-1}&0\cr 0&-s\cr \end{pmatrix}$.

\bigskip

Now let $t^\prime = \begin{pmatrix} r&0\cr
0&r^{-1}\cr \end{pmatrix}.$ Suppose that $s, r \ne \pm 1$ then the centralizer of  the matrix $t$ and also the matrix $ t^\prime$ is $T$. From \ref{equahhjhhjhhhhfhfs1} a smallest set of representatives of $G$-orbits (under conjugation) of $C_t\times C_{t^\prime}$, where $C_t,  C_{t^\prime}$ are conjugacy classes of $t, t^\prime$ respectively, consists of the following pairs:
$$(g_\a t g_\a^{-1}, t^\prime)_{\a\in K},\,\,(\dot w(1) t\dot w(1)^{-1}, t^\prime), \,\,\,(\dot w t\dot w^{-1}, t^\prime),\,\,(u(1)tu(1)^{-1}, t^\prime), \,\,\,(t, t^\prime).$$
Thus we have the folllowing
\begin{prop}
\label{prhhjhfffffs1}
Let $\pm \dot e \ne t, t^\prime \in T$ and let $C_t, C_{t^\prime}$ be corresponding conjugacy classes. Then there are only the following $G$-orbits on $C_t \times C_{t^\prime}$:
$$\mathcal O_\a:=\{A\Big(\begin{pmatrix} s + \a\Delta_t &- \a\Delta_t\cr
(1+\a)\Delta_t& \,\,\,\,s^{-1} - \a\Delta_t\cr \end{pmatrix},\,\begin{pmatrix} r&0\cr
0&r^{-1}\cr \end{pmatrix}\Big)A^{-1}\,\,\mid\, \, \a\in K,\,\,A \in \SL_2(K)\}$$
(the orbits of  $(g_\a t g_\a^{-1}, t^\prime)_{\a\in K}$);
$$\mathcal O_U^+:= \{A \Big(\begin{pmatrix} s  &\Delta_t\cr
0& \,\,\,\,s^{-1} \cr \end{pmatrix},\,\,\begin{pmatrix} r&0\cr
0&r^{-1}\cr \end{pmatrix}\Big)A^{-1}\,\,\,\mid\,\,\,A\in \SL_2(K)\}$$
 (the orbit $(u(1)tu(1)^{-1}, t^\prime)$);
$$\mathcal O_V^-:= \{A \Big(\begin{pmatrix} s^{-1}  &0\cr
\Delta_t& \,\,\,\,s \cr \end{pmatrix},\,\,\begin{pmatrix} r&0\cr
0&r^{-1}\cr \end{pmatrix}\Big)A^{-1}\,\,\,\mid\,\,\,A\in \SL_2(K)\}$$
(the orbit of $(\dot w(1) t\dot w(1)^{-1}, t^\prime)$;
$$\mathcal O_T^+:= \{A \Big(\begin{pmatrix} s  &0\cr
0& \,\,\,\,s^{-1} \cr \end{pmatrix},\,\,\begin{pmatrix} r&0\cr
0&r^{-1}\cr \end{pmatrix}\Big)A^{-1}\,\,\,\mid\,\,\,A\in \SL_2(K)\}$$
(the orbit of $(t, t^\prime)$);
$$\mathcal O_T^-:= \{A \Big(\begin{pmatrix} s^{-1}  &0\cr
0& \,\,\,\,s \cr \end{pmatrix},\,\,\begin{pmatrix} r&0\cr
0&r^{-1}_1\cr \end{pmatrix}\Big)A^{-1}\,\,\,\mid\,\,\,A\in \SL_2(K)\}$$
(the orbit of $(\dot w t\dot w^{-1}, t^\prime)$).

\end{prop}

\bigskip

{\it Adherence of orbits.}  Consider the adherence of  $G$-orbits on $C_t\times C_{t^\prime}$. Note, that the invariant algebra of $\Me_2(\C)\times \Me_2(\C)$, which corresponds to the action by conjugation of $\SL_2(\C)$, is generated  by $\tr X, \tr Y, \tr XY$ where $(X, Y) \in \Me_2(\C)\times \Me_2(\C)$ (see, for instance, \cite{VP}, 9.5).  The algebraic factor $\Big(\Me_2(\C)\times \Me_2(\C)\Big)/\SL_2(\C)$ is isomorphic to $A_\C^3$. Thus we have the quotient  map $\pi: \Me_2(\C)\times \Me_2(\C)\rightarrow A_\C^3$ where $\pi: (X, Y) = (\tr X, \tr XY, \tr Y)$. In every fiber of this map there is only one closed orbit.

Now we consider the restriction of $\pi_{t, t^\prime}$ on the closed subset $C_t\times C_{t^\prime}\subset\Me_2(\C)\times \Me_2(\C)$ (recall, that the conjuagcy class of a semisimple element is a closed subset of $\Me_2(\C)$). Since $\tr X, \tr Y$ are constants on $C_t\times C_{t^\prime}$ we will consider the map $\pi_{t, t^\prime}$ as  $$\pi_{t, t^\prime}: C_t\times C_{t^\prime} \rightarrow A^1_\C = \C \,\,\,\text{where}\,\,\,\pi_{t, t^\prime}((X, Y)) = \tr(XY).$$
Every fiber of $\pi_{t, t^\prime}$ contains only one closed $G$-orbit in $C_t\times C_{t^\prime}$. Also, every fiber of $\pi_{t, t^\prime}$ has the dimension $\geq 3$ (indeed, $\dim C_t\times C_{t^\prime}  - \dim \Imm \pi_{t, t^\prime} = 3$).

\bigskip

Consider   an orbit of the type $\mathcal O_\a$.

Let $(X_\a, Y)\in \mathcal O_\a$ be the representative which is pointed out in Proposition \ref{prhhjhfffffs1}.

Here
\begin{equation}
\label{equajjhjhhhjhhfs1}
\tr(X_\a Y) = \tr \Big(\begin{pmatrix} s + \a\Delta_t &- \a\Delta_t\cr
(1+\a)\Delta_t& \,\,\,\,s^{-1} - \a\Delta_t\cr \end{pmatrix}\begin{pmatrix} r&0\cr
0&r^{-1}\cr \end{pmatrix}\Big) = sr + s^{-1}r^{-1} +\a\Delta_t\Delta_{t^\prime}.
\end{equation}
Further, let $Y= \begin{pmatrix} r&0\cr
0&r^{-1}\cr \end{pmatrix}$. Then for every $a \in K$
\begin{equation}
\label{equahhjhhjhhjs1}
\tr(XY) = \begin{cases}  sr + s^{-1}r^{-1}\,\,\,\text{if}\,\, X = \begin{pmatrix} s  &a\cr
0& \,\,\,\,s^{-1} \cr \end{pmatrix},\,\,\begin{pmatrix} s  &0\cr
a& \,\,\,\,s^{-1} \cr \end{pmatrix},\\
$\,$\\
s^{-1}r + sr^{-1}\,\,\,\,\,\text{if}\,\, X = \begin{pmatrix} s^{-1}  &a\cr
0& \,\,\,\,s \cr \end{pmatrix},  \,\,\begin{pmatrix} s^{-1}  &0\cr
a& \,\,\,\,s \cr \end{pmatrix}.\end{cases}.
\end{equation}
Note,
$$ (sr + s^{-1}r^{-1}) - (s^{-1}r + sr^{-1}) = (s - s^{-1})(r - r^{-1}) = \Delta_t\Delta_{t^\prime} \ne 0\Rightarrow$$

\begin{equation}
\label{equahhjhhhhhhhjjhs1}
sr + s^{-1}r^{-1}  + \underbrace{(-1) }_{= \a} \Delta_t\Delta_{t^\prime} = (s^{-1}r + sr^{-1})
\end{equation}
Let  $\a \ne 0,\, -1$.  From \ref{equajjhjhhhjhhfs1}, \ref{equahhjhhhhhhhjjhs1} we get
\begin{equation}
\label{equahhjhhh21hhhhjjhs1}
 \tr(X_\a Y) \ne sr + s^{-1}r^{-1},\,\,\,s^{-1}r + sr^{-1}.
 \end{equation}
 Now \ref{equahhjhhjhhjs1}, \ref{equahhjhhh21hhhhjjhs1} imply
that $G$-orbits $\mathcal O_U^+, \,\mathcal O_T^+,\,\mathcal O_V^-,\,\,\mathcal O_T^-$   cannot be in the same fiber of $\pi_{t,\, t^\prime}$  that contains $(X_\a, Y)$. Since the stabilizer of $(X_\a, Y)$ is equal to $\{\pm \dot e\}$, the dimension of $G$-orbit is equal to 3. Hence the orbit  of $(X_\a, Y)$ is closed and coincides with the fiber of $\pi_{t, t^\prime}$ which contains $(X_\a, Y)$.

Consider the orbits $\mathcal O_{0}, \mathcal O_U^+$. Both orbits have   dimension 3 and belong to the fiber $\pi^{-1}(sr +s^{-1}r^{-1})$ (see, \ref{equahhjhhjhhjs1}). Moreover, in the same fiber we have 2-dimensional closed orbit $\mathcal O_T^+$ and
$$\overline{\mathcal O_{0}}\setminus \mathcal O_{0} = \overline{\mathcal O_U^+}\setminus \mathcal O_U^+ = \mathcal O_T^+.$$
Indeed, for any fixed $0\ne a, b \in K$ we have
$$\overline{\{\begin{pmatrix}a& c\cr 0&b\cr \end{pmatrix}\,\,\, \mid\,\, 0\ne c \in K\}} \setminus \{\begin{pmatrix}a& c\cr 0&b\cr \end{pmatrix}\,\,\, \mid\,\, 0\ne c \in K\} = \{\begin{pmatrix} a&0\cr 0&b\cr\end{pmatrix}\}$$
(here $\overline{X}$ is the closure in Zariski topology). The same equality hold  for lower triangular matrices.

The analogical  result we can get for $\mathcal O_{-1}, \mathcal O_V^-, \mathcal O_T^-$.

Now we summarize the facts on adherence of orbits
\begin{prop}

$\, $

i. The orbits $\mathcal O_\a,$ where $\a \ne 0, -1,$ are closed 3-dimensional $G$-orbits which coinside with fibers $\pi_{t, t^\prime}^{-1}(l_\a)$ for
$l_\a = \tr (X_\a, Y) \ne  sr + s^{-1}r^{-1},  s^{-1}r + sr^{-1}$.

ii. The fiber $\pi_{t, t^\prime}^{-1}(l_0)$ where $l_0 = \tr (X_0 Y) =  sr + s^{-1}r^{-1}$ consists of  two 3-dimensional orbits
$$\mathcal O_{0} = \{A\Big(\begin{pmatrix} s &0\cr
\Delta_t& \,\,\,\,s^{-1} \cr \end{pmatrix},\,\begin{pmatrix} r&0\cr
0&r^{-1}_1\cr \end{pmatrix}\Big)A^{-1}\,\,\mid\, \, \,\,A \in \SL_2(K)\},$$
$$\mathcal O_U^+ = \{A\Big(\begin{pmatrix} s &-\Delta_t\cr
0& \,\,\,\,s^{-1} \cr \end{pmatrix},\,\begin{pmatrix} r&0\cr
0&r^{-1}_1\cr \end{pmatrix}\Big)A^{-1}\,\,\mid\, \,\,A \in \SL_2(K)\}$$
and the closed 2-dimensional orbit
$$\mathcal O_T^+ = \{A\Big(\begin{pmatrix} s &0\cr
0& \,\,\,\,s^{-1} \cr \end{pmatrix},\,\begin{pmatrix} r&0\cr
0&r^{-1}_1\cr \end{pmatrix}\Big)A^{-1}\,\,\mid\, \,\,A \in \SL_2(K)\}$$
which  coincides with $\overline{\mathcal O_{0}}\cap  \overline{\mathcal O_U^+}$.

iii. The fiber $\pi_{t, t^\prime}^{-1}(l_{-1})$ where $l_{-1} = \tr (X_{-1} Y) =  s^{-1}r + sr^{-1}$ consists of  two 3-dimensional orbits
$$\mathcal O_{-1} = \{A\Big(\begin{pmatrix} s^{-1} &\Delta_t\cr
0& \,\,\,\,s \cr \end{pmatrix},\,\begin{pmatrix} r&0\cr
0&r^{-1}_1\cr \end{pmatrix}\Big)A^{-1}\,\,\mid\, \, \,\,A \in \SL_2(K)\},$$
$$\mathcal O_V^- = \{A\Big(\begin{pmatrix} s^{-1} &0\cr
\Delta_t& \,\,\,\,s \cr \end{pmatrix},\,\begin{pmatrix} r&0\cr
0&r^{-1}_1\cr \end{pmatrix}\Big)A^{-1}\,\,\mid\, \,\,A \in \SL_2(K)\}$$
and the closed 2-dimensional orbit
$$\mathcal O_T^- = \{A\Big(\begin{pmatrix} s^{-1} &0\cr
0& \,\,\,\,s \cr \end{pmatrix},\,\begin{pmatrix} r&0\cr
0&r^{-1}_1\cr \end{pmatrix}\Big)A^{-1}\,\,\mid\, \,\,A \in \SL_2(K)\}$$
which  coincides with $\overline{\mathcal O_{-1}}\cap  \overline{\mathcal O_V^-}$.
\end{prop}

\end{document}